\documentclass[reqno]{amsart}

\usepackage{amsmath,amssymb,amscd, accents} 
\usepackage{mathrsfs}
\usepackage[mathcal]{eucal} 

\newtheorem{Thm}{Theorem}{\bfseries}{\itshape}
\newtheorem*{Thm*}{Theorem}{\bfseries}{\itshape}
\newtheorem{Cor}{Corollary}{\bfseries}{\itshape}
\newtheorem{Prop}[Cor]{Proposition}{\bfseries}{\itshape}
\newtheorem{Lem}[Cor]{Lemma}{\bfseries}{\itshape}
{\bfseries}{\itshape}
\newtheorem{Def}[Cor]{Definition}{\bfseries}{\rmfamily}
\newtheorem{Ex}[Cor]{Example}{\scshape}{\rmfamily}
\newtheorem{Rem}[Cor]{Remark}{\scshape}{\rmfamily}
{\scshape}{\rmfamily}

\renewcommand\ge{\geqslant} \renewcommand\le{\leqslant}
\let\tildeaccent=\~ \let\hataccent=\^
\renewcommand\~[1]{\widetilde{#1}}

\def\<{\left<} \def\>{\right>} \def\({\left(} \def\){\right)}

\def\abs#1{\left\vert #1 \right\vert} \def\norm#1{\left\Vert #1
  \right\Vert}

 \def\pd#1#2{\frac{\partial#1}{\partial#2}}

\let\polishL=l \def\Zoladek.{\.Zol\c adek}

 \def\Im{\operatorname{Im}}
 \def\dist{\operatorname{dist}}
 \def\ord{\operatorname{ord}}
 \def\etc.{\emph{etc}.}

\def\iclo#1{\overline{#1}}

\def\:{\colon} \def\R{{\mathbb R}} \def\C{{\mathbb C}}  \def\N{{\mathbb N}} \def\Q{{\mathbb Q}} 
 \def\e{\varepsilon}

 \def\d{\,\mathrm d}

 \let\PolishL=\L \def\Lojas.{\PolishL ojasiewicz}
 
\def\cP{{\mathcal P}}  
 \def\cL{{\mathcal L}}

\def\cO{{\mathcal O}}

 \def\mult{\operatorname{mult}}

\def\rest#1{{\vert_{#1}}}

\def\LT{\operatorname{LT}} \def\spec{\operatorname{Spec}}

\def\bfi{{\mathbf i}}

\def\CD{C^D}  \def\CZ{C^Z}

\newcommand{\mo}[1][k]{M^{\smash{(#1)}}}
\newcommand{\bmo}{M}

\begin{document}

\title{Multiplicity Operators} 
\author{Gal Binyamini}\address{University of Toronto, Toronto, 
Canada}\email{galbin@gmail.com}
\thanks{The first author was supported by the Banting
  Postdoctoral Fellowship and the Rothschild Fellowship}
\author{Dmitry Novikov}\address{Weizmann Institute of Science, Rehovot, 
Israel}\email{dmitry.novikov@weizmann.ac.il}
\thanks{Supported by the Minerva foundation with funding from the Federal 
German Ministry for Education and Research.}

\begin{abstract}
  For functions of a single complex variable, zeros of multiplicity
  greater than $k$ are characterized by the vanishing of the first $k$
  derivatives. There are various quantitative generalizations of this
  statement, showing that for functions that are in some sense close
  to having a zero of multiplicity greater than $k$, the first $k$
  derivatives must be small.

  In this paper we aim to generalize this situation to the
  multi-dimensional setting. We define a class of differential
  operators, the \emph{multiplicity operators}, which act on maps from
  $\C^n$ to $\C^n$ and satisfy properties analogous to those described
  above. We demonstrate the usefulness of the construction by applying
  it to some problems in the theory of Noetherian functions.
\end{abstract}
\maketitle
\date{\today}

\section{Introduction}

The derivative operator is an invaluable tool in the study of zeros of
functions of a single variable. In the real setting, the Rolle theorem
roughly states that if a function admits $k$ zeros on an interval, its
derivative admits at least $k-1$ zeros. This statement and its
numerous ramifications have found applications in the qualitative
theory of ordinary differential equations (see for instance the survey
\cite{yakovenko:survey} and references therein), and in the theory of
Fewnomials and Pfaffian functions (see the monograph \cite{Kho:Few}
and the paper \cite{Gab:Loj}).

In the complex setting, the literal analog of the Rolle theorem fails.
However, certain local analogs still hold and can be used to study the
zeros of holomorphic functions (in one variable). The goal of this
paper is to provide multi-dimensional analogs of the derivative
operator in this context, and to illustrate their application in the
study of Noetherian systems of equations, their multiplicities and the
number of their solutions. In the remainder of this section we review
some of the classical one-dimensional results in this spirit and
reference results in the body of the paper which generalize them to
several variables.

We fix some notations used throughout the paper. We let $D_r$ denote
the closed complex disc of radius $r$ and $D^n_r$ (resp. $B^n_r$)
denote the corresponding polydisc (resp. ball) in $\C^n$. When $r$ is
omitted it is assumed that $r=1$. For a complex domain $U$ (in any
dimension) we let $\cO(U)$ denote the space of analytic functions on
$U$, and $\norm{\cdot}$ the maximum norm on $\cO(U)$. Finally, we let
$\cO^n(U)$ denote the space of analytic maps into $\C^n$, and
$\norm\cdot$ the direct sum norm.

A function of a single complex variable admits a zero of multiplicity
greater than $k$ if and only if its derivatives vanish up to order
$k$. It is natural to expect that a quantitative generalization of
this statement also holds -- namely, that if a function is ``close''
to having a zero of multiplicity greater than $k$ then its derivatives
up to order $k$ must be small. Below we give two quantitative
statements of this form.

\begin{Thm*}[I]
  Let $f\in\cO(D)$ and assume that $\norm{f}\le1$. There exists an
  absolute positive constant $\CZ_{1,k}$ such that if $f$ has more
  than $k$ zeros in $D_r$ then
  \begin{equation}
    \CZ_{1,k} \abs{\smash{f^{(k)}(0)}} < r
  \end{equation}
\end{Thm*}

This theorem shows that if a function has a $k$-th derivative of size
$s$ then the function cannot have more than $k$ zeros in a disc whose
radius is proportional to $s$. It therefore gives control on the
number of zeros of a function in terms of its derivatives. A
generalization of this theorem to several variables is given in
Theorem~\ref{thm:polydisc-zeros}.

\begin{Thm*}[II]
  Let $f\in\cO(D)$ and assume that $\norm{f}\le1$. There exist
  absolute positive constants $A_{1,k},B_{1,k}$ with the following
  property:

  For every $r < s=|f^{(k)}(0)|$ there exists $A_{1,k} r < 
  \tilde r < r$
  such that
  \begin{equation}
    \abs{f(z)} \ge B_{1,k} s \tilde r^k \mbox{ for every } \abs{z}=\tilde r
  \end{equation}
\end{Thm*}

This theorem shows that if a function has a $k$-th derivative of size
$s$ then up to a constant, the function grows at least like $sr^k$ on
spheres of radius $r<s$. This is not true for every such sphere (for
instance, the function might have zeros on some of them), but it is
true if one is allowed to move the sphere within some constant
multiplicative range. A generalization of this theorem to several
variables is given in Theorem~\ref{thm:sphere-growth}.

In conjunction with Rouch\'e's theorem, this theorem allows to conclude
that a perturbation of size $\e<s$ doesn't change  the number of zeros of $f$ 
in a disk of radius $(\e/s)^{1/k}$ (up to a multiplicative
constant). A precise statement of this type is given in
Corollary~\ref{cor:perts}.

The derivative operator also has a controlled decreasing effect on the
multiplicities of functions and ideals. We summarize various simple
formulations of this fact below.

\begin{Thm*}[III] The following statements hold:
  \begin{enumerate}
  \item Let $f\in\C((z))$. Then $\mult_0 f^{(k)}\ge\mult_0f-k$.
  \item Let $f$ be a real analytic function on $[0,1]$. Then $\ord_0
    f^{(k)}\ge\ord_0f-k$.
  \item Let $I\subset\C((z))$ be an ideal. Denote by $I^{(k)}$ the
    ideal generated by derivatives up to order $k$ of elements of
    $I$. Then $t^k I^{(k)}\subset I$ and $\mult I^{(k)}\ge\mult I-k$.
  \end{enumerate}
\end{Thm*}

Various generalizations of these statements to several variables are
given in Theorems~\ref{thm:curve-growth}~and~\ref{thm:mult-op-iclo}
and in Corollary~\ref{cor:mult-up-ideal-mult}.

\section{Multiplicity Operators}

Let $R_n=\C[[x_1,\ldots,x_n]]$ denote the ring of formal power series
in $n$ variables, and $f_1,\ldots,f_n\in R_n$. Denote
$F=(f_1,\ldots,f_n)$.  We denote by $I(F)$ the ideal generated by
$f_1,\ldots,f_n$. Recall that the multiplicity of the common zero
$f_1=\ldots=f_n=0$, denoted $\mult F$, is defined to be $\dim_{\C}
R_n/I(F)$.

Since our estimates involve many constants, we use the convention that
constants appearing with a numeric superscript, e.g. $C^1$, may appear
in different proofs to denote different numbers. Constants with
non-numeric superscripts are meant to be defined throughout the paper.

Let $k$ be a non-negative integer.  Let $\frak m$ denote the maximal
ideal in $R_n$, let $J_{n,k}=R_n/{\frak m}^{k+1}$ denote the ring of
$k$-jets in $n$ variables, and let $j:R_n\to J_{n,k}$ denote the
canonical homomorphism. We also fix an arbitrary norm, for instance
the $l_1$ norm, on the space of polynomials.

We are interested in writing down differential-algebraic conditions on
$F$ that are satisfied if and only if $\mult F > k$. We begin by
recalling the following standard lemma.

\begin{Lem} \label{lem:jet-mult} We have $\mult F>k$ if and only if
  $\dim_{\C} J_{n,k}/j(I(F)) >k$
\end{Lem}
\begin{proof}
  First, since
  \begin{equation} \label{eq:jet-mult} J_{n,k}/j(I(F)) \simeq
    R_n/(I(F)+{\frak m}^{k+1})
  \end{equation}
  it is clear that $\dim_{\C} J_{n,k}/j(I(F))>k$ implies $\mult F>k$.
  Conversely, suppose that $\dim_{\C} J_{n,k}/j(I(F))\le k$. Then it
  known that $j({\frak m}^{k+1})\subset j(I(F))$ (see, for instance,
  \cite[Lemma, Section 5.5]{AVG}). Since $R_n$ is complete with
  respect to the $\frak m$-filtration, it follows from
  \cite[Proposition 7.12]{Eis:book} that
  ${\frak m}^{k+1}\subset I(F)$. Then it follows
  from~\eqref{eq:jet-mult} that $\dim_{\C} J_{n,k}/j(I(F))=\mult F$.
\end{proof}

\subsection{Construction of the multiplicity operators}

For $\bfi\in\N^n$ let $x^\bfi$ denote the corresponding monomial in
$R_n$.  For any co-ideal $B\subset\N^n$ we call the set
$x^B:=\{x^{\mathbf b}:b\in B\}$ a standard monomial set. We recall
another standard fact from Gr\"obner base theory.

\begin{Lem} \label{lem:standard-base} For any ideal
  $I\subset J_{n,k}$, the ring $J_{n,k}/I$ has a standard monomial
  base. 
\end{Lem}
\begin{proof}
  Let $\LT(I)\subset\N^n$ denote the ideal generated by monomials
  which are leading terms of members of $I$, with respect to some
  monomial ordering. Then the monomials outside of $\LT(I)$ form a
  base for the ring $J_{n,k}/I$.
\end{proof}

A standard monomial set $x^B$ of size $k$ spans $J_{n,k}/j(I(F))$ if and
only if $J_{n,k}=\C\<x^B\>+j(I(F))$. We encode this information in the
linear map
\begin{equation}
  T^F_B : \C^{\abs B}\oplus J_{n,k}^{\oplus n} \to J_{n,k} \qquad
  (c_\bfi)_{\bfi\in B} \oplus u_1 \oplus \ldots \oplus u_n  \to
  \sum_{\bfi\in B} c_\bfi x^\bfi + \sum_i u_i j(f_i).
\end{equation}
For explicitness, we represent $T^F_B$ as a matrix with respect to the
basis consisting of $c_\bfi,\bfi\in B$ followed by the basis of
monomials for $J_{n,k}^{\oplus n}$. The set $x^B$ spans
$J_{n,k}/j(I(F))$ if and only if one of the maximal minors containing
the first $\abs{S}$ columns is non-zero.

The entries of the first $k$ columns of $T^F_B$ are constants (in
fact, each has one component $1$ and zero elsewhere). For the
remaining columns, each component is a Taylor coefficient of the map
$F$ of order up to $k$. We thus view minors of the matrix as
differential operators of order bounded by $k$.

\begin{Def}
  Let $\{\bmo_B^\alpha\}$ be the set of maximal minors of $T^F_B$ which
  contain the first $\abs B=k$ columns. We call each $\bmo_B^\alpha$ a
  \emph{basic multiplicity operator} of order $k$.
  We call any operator $\mo$ in the convex hull of this set
  a \emph{multiplicity operator} of order $k$.
\end{Def}

Each multiplicity operator $\mo$ is a differential operator of order
$k$ of homogeneity degree $\dim J_{n,k} - k={n+k \choose k}-k$.
Therefore when studying the behavior of multiplicity operators on
analytic maps $F\in\cO^n(D^n)$, there will be no loss of generality in
assuming $\norm{F}\le1$, and we will usually restrict to this case.

\begin{Ex}
For $F:(\C,0)\to(\C,0)$, $F(x)=\eta x+x^2$ the multiplicity operators are 
$\bmo_{\{1\}}^{\alpha}=\eta$ and $\bmo_{\{1,x\}}^{\alpha}=1$ for $x^B=\{1\}$ 
and $x^B=\{1,x\}$ correspondingly.
\end{Ex}
 To simplify the
notations, we denote by $\mo_p$ the differential functional obtained by
applying $\mo$, followed by evaluation at the point $p$.

The following proposition, a direct conclusion of the discussion
above, is the basic property of the multiplicity operators.

\begin{Prop}
  We have $\mult F_p>k$ if and only if $\mo_p(F)=0$ for all multiplicity
  operators of order $k$.
\end{Prop}

\subsection{Effective decomposition in the local algebra}

The following proposition gives the main technical property of the
multiplicity operators that will be used throughout this paper.

\begin{Prop} \label{prop:decomposition} Let
  $F\in\cO^n(D^n)$ with $\norm{F}\le1$. Let $\bmo_B^\alpha$ be a basic
  multiplicity operator of order $k$. Assume
  $s=|\bmo_B^\alpha(F)(0)|>0$. Then for any $P\in J_{n,k}$ we have a
  decomposition
  \begin{equation} \label{eq:decomposition} 
  P = \sum_{\bfi\in B}
    c_\bfi x^\bfi + \sum U_i f_i + E, \quad E\in {\frak m}^{k+1}
  \end{equation}
  with
  \begin{equation*}
    \max(\abs{c_\bfi},\norm{U_i},\|E\|) \le \CD_{n,k} s^{-1} 
    \norm{P},
  \end{equation*}
 where $\CD_{n,k}$ is some universal constant.
\end{Prop}
\begin{proof}
  The existence of such a decomposition with the corresponding
  estimates on $\abs{c_\bfi}$ and $\norm{U_i}$ is a direct consequence of Cramer's rule
  applied to the maximal minor of the matrix $T^F_B$ corresponding
  to $\bmo_B^\alpha$ (this gives~\eqref{eq:decomposition} without the $E$ term
  in the ring $J_{n,k}$).

  Since $\norm{F}\le1$, the Taylor coefficients of order $k$ of $F$
  are bounded by some universal constant (for instance by the Cauchy
  formula). Thus we may also assume that $s$ is bounded by a
  universal constant independent of $F$. The estimate on $\norm{E}$
  then follows, with an appropriate universal constant $\CD_{n,k}$,
  from
  \begin{equation}
    E = P - \sum_{\bfi\in B} c_\bfi x^\bfi - \sum U_i f_i
  \end{equation}
  using the triangle inequality.
\end{proof}

The following is an important corollary of
Proposition~\ref{prop:decomposition}. 

\begin{Lem} \label{lem:local-resultant}
  Let $F\in\cO^n(D^n)$ with $\norm{F}\le1$. Suppose a
  $k$-th multiplicity operator $\mo_0(F)$ is nonzero. Let
  $p_0,...,p_k\in J_{n,k}$ be some linearly independent polynomials of
  unit norm. Then there exists their linear combination
  $P=\sum c_ip_i$ of degree $k$ such that $\sum|c_i|=1$ and
  \begin{equation*}
    P = \sum_i U_i f_i + E
  \end{equation*}
  where
  \begin{equation*}
    \max(\norm{U_i},\norm{E})<\CD_{n,k} \abs{\mo_0(F)}^{-1} \text{ and }
    E\in{\frak m}^{k+1}.
  \end{equation*}
\end{Lem}
\begin{proof}
  It is clearly sufficient to prove the statement for every basic
  multiplicity operator $\mo=\bmo_B^\alpha$.
  
  According to Proposition~\ref{prop:decomposition} we have the
  following decompositions:
  \begin{equation*}
    p_j = \sum_{\bfi\in B} c_\bfi^j x^\bfi + \sum U_i^j f_i + 
   E_j, \quad j=0,...,k,
  \end{equation*}
  with $\norm{U_i^j}\le \CD_{n,k}\abs{\mo_0(F)}^{-1}$, $\norm{E_j}\le 
  C^{D}_{n,k}\abs{\mo_0(F)}^{-1}$ and $E_j\in{\frak m}^{k+1}$.
  
  Since $\abs{B}=k$, there is a non-trivial linear combination of
  these equations for which all of the $x^\bfi$ terms vanish.
  Normalizing the coefficients to $\sum|c_i|=1$, we obtain the
  required decomposition and estimates.
\end{proof}

\begin{Rem}
  Note that, in general, even if $g\in I(F)$ one cannot claim that
  there is a decomposition $g=\sum u_if_i$ with bounds as in
  Lemma~\ref{lem:local-resultant}. E.g. for $F=\eta x+x^2$ one has
  $\bmo_{\{1,x\}}^\alpha(F)=1$, but for $g=x$ the coefficients of the
  decomposition $g=(\eta^{-1}-\eta^{-2}x)F$ explode as $\eta\to 0$,
  whereas the coefficients of the decomposition
  $g=0\cdot 1+1\cdot x +0\cdot F$ of Lemma~\ref{prop:decomposition}
  remain bounded.
\end{Rem}

When $\bmo_B^\alpha(F)(0)\neq0$, Proposition~\ref{prop:decomposition}
gives estimates on the corresponding decompositions in the ring of
$k$-jets $J_{n,k}$. However, by Lemma~\ref{lem:jet-mult} in is natural
to expect that similar decompositions exist in the local ring of germs
of holomorphic functions. Our next goal is the corresponding
quantitative result.

\subsection{Effective division by $F$ with remainder in $x^B$}

Let $R_{n,t}$ be the linear subspace of $R_n$ consisting of series
with finite $\|\;\|_t$ norm, with
$\|\sum c_\alpha x^\alpha\|_t=\sum t^{|\alpha|}|c_\alpha|$.

Note that if $f\in\cO(D^n)$ and $\|f\|\le 1$ then absolute values of
all Taylor coefficients of $f$ are bounded by $1$, and therefore
$\|f\|_{1/2}\le 2^n$. Therefore, if $j^k(f)=0$ then
\begin{equation}\label{eq:norm equivalence}
  \|f\|_t\le (2t)^{k+1}\|f\|_{1/2}\le t^{k+1}2^{k+n+1}\|f\|
\end{equation}
for $0<t<1/2$.

\begin{Lem}\label{lem:weierstrass}
  Let $F\in\cO^n(D^n)$ with $\norm{F}\le1$. Assume
  $s=|\bmo^\delta_B(F)(0)|>0$.

  There exist positive universal
  constants $\epsilon_{n,k},\epsilon'_{n,k}, C_{n,k}>0$ such that for
  some $t$, $\epsilon'_{n,k}s\le t\le \epsilon_{n,k}s$, and for every
  $P\in R_{n,t}$ there exists a decomposition
  \begin{equation}\label{eq:weierstrass}
    P=\sum_{i=1}^nu_if_i+\sum_{\bfi\in B}
    c_\bfi x^\bfi,
  \end{equation}
  with
  \begin{equation}\label{eq:weierstrass bounds}
    \sum \|u_i\|_t + \|\sum_{\bfi\in B}
    c_\bfi x^\bfi\|_t\le C_{n,k}s^{-k-1}\|P\|_t.
  \end{equation}
\end{Lem}

The strategy of the proof is as follows. Let $A>2$ be fixed
(eventually $A=3$).

We begin by showing that for some $t$ in the prescribed range, and for
every monomial $x^\alpha,|\alpha|=k$, there exists a decomposition
\begin{equation} 
  x^{\alpha}=P_{\alpha}(x)+\sum u_{i,\alpha}f_i+E'_\alpha,\qquad\deg 
  P_\alpha(x)<k, \quad E'_\alpha\in{\frak m}^{k+1}
\end{equation}
with
\begin{equation}
  \|P_\alpha\|_t+\|E'_\alpha\|_t<A^{-1}\|x^\alpha\|_t.
\end{equation}
By multiplying this decomposition by an arbitrary monomial and
applying~\eqref{eq:decomposition} we get for any monomial
$x^{\alpha}\in{\frak m}^{k+1}$, and, by linearity, for any
$f\in {\frak m}^{k+1}_t$ the decomposition
\begin{equation}\label{eq:operator decomposition}
  f=\pi_B(f)+U(f)\cdot(f_1,...,f_n)^T+E(f),
\end{equation}
where
\begin{equation}
  \pi_B:{\frak m}^{k+1}_t\to\bigoplus_{\bfi\in B}\C x^\bfi,\quad
  U:{\frak m}^{k+1}_t\to R_{n,t}^{\oplus n},\quad
  E: {\frak m}^{k+1}_t\to {\frak m}^{k+1}_t
\end{equation}
are some explicitly bounded linear operators, and the ``remainder
operator'' $E$ has norm $2A^{-1}<1$. Therefore
\begin{equation}\label{eq:operator decomposition2}
  f=\pi_B(1-E)^{-1}(f)+U(1-E)^{-1}(f)\cdot(f_1,...,f_n)^T,
\end{equation}
which is exactly \eqref{eq:weierstrass} for elements of 
$\frak{m}^{k+1}$. The general case then follows easily from 
Proposition~\ref{prop:decomposition}.

\begin{proof}
  Let $\epsilon_{n,k}^{-1}=2^{n+k+1}C^{D}_{n,k}$ and let
  $C_{n,k}(A)^{-1}=(2A+1)^{\binom{n+k-1}{k}2(k+1)}(k+1)$.

\begin{Lem}\label{lem:lem1ofDecomp}
  For any $A>1$ there exists some $t$,
  \begin{equation}\label{eq: t bounds}
    C_{n,k}(A)\epsilon_{n,k}s\le t \le \epsilon_{n,k}s,
  \end{equation}
  such that any monomial $x^{\alpha}$, $|\alpha|=k$ can be represented
  as
  \begin{equation}\label{eq:Monomial representation}
    x^{\alpha}=P_{\alpha}(x)+\sum u_{i,\alpha}f_i+E'_\alpha,\qquad\deg 
    P_\alpha(x)<k, \quad E'_\alpha\in{\frak m}^{k+1}.
  \end{equation}
  with
  \begin{equation}\label{eq:decomposition bounds}
    \|P_\alpha\|_t+\|E'_\alpha\|_t<A^{-1}\|x^\alpha\|_t, \quad 
    \|u_{i,\alpha}\|_t\le 2C^D_{n,k}s^{-1}t^{-k} \|x^\alpha\|_t.
  \end{equation} 
\end{Lem}

\begin{proof}
  Let $\{x^{\alpha_i}\}_{i=0}^k$ be a chain of length $k+1$ of monomials
  ending with $x^\alpha$, such that each monomial is divisible by the
  previous one (so $\deg x^{\alpha_i}=i$).

  By Lemma~\ref{lem:local-resultant} there exists some linear
  combination $\sum_{i=0}^k c_{\alpha,i}x^{\alpha_i}$ with
  $\sum|c_i|=1$ such that
  \begin{equation}
    \sum_{i=0}^k
    c_{\alpha,i} x^{\alpha_i}= \sum u_{\alpha,i} f_i + E_\alpha,\quad 
    E_\alpha\in m^{k+1},\ \deg u_{\alpha,i}\le k
  \end{equation}
  and
  \begin{equation}
    \max_{i,\alpha}\|u_{\alpha,i}\|\le C^D_{n,k}s^{-1}, \quad \|E_\alpha\|\le 
    C^{D}_{n,k}s^{-1}.
  \end{equation}
  Here the norm is the supremum norm on the unit polydisc.

  Apply Lemma~\ref{lem:Valiron} below to the set of the sequences
  $\{|c_{\alpha,0}|,...,|c_{\alpha,k}|,2^{n+k+1}\|E_\alpha\|\}$, with
  $M=2^{n+k+1}C^{D}_{n,k}s^{-1}$ and $t_0=\epsilon_{n,k}s$. As
  \begin{equation}\label{eq:epsilon condition}
    M(k+1)>\left(\epsilon_{n,k}s\right)^{-1}
  \end{equation}
  we conclude existence of some $t$ satisfying \eqref{eq: t bounds}
  and some $j(\alpha)$ such that
  \begin{align}\label{eq:decomp bounds1}
    \|c_{\alpha,j(\alpha)}x^{\alpha_{j(\alpha)}}\|_t &\ge
     A\bigg[\|\sum_{i\not=j}c_{\alpha,i}x^{\alpha_i}\|_t+t^{k+1}2^{k+n+1}
      \|E_\alpha\|\bigg]\
    \\ \nonumber
    & \ge 
    A\bigg[\|\sum_{i\not=j}c_{\alpha,i}x^{\alpha_i}\|_t+\|E_\alpha\|_t\bigg],
  \end{align}
  and
  \begin{align}\label{eq:decomp bounds2}
    \|u_{\alpha,i}\|_t &\le \|u_{\alpha,i}\|_1\le  C^D_{n,k}s^{-1}\|\sum 
    c_{\alpha,i}x^{\alpha_i}\|_1\le 
    C^D_{n,k}s^{-1}t^{-k}\|\sum c_{\alpha,i}x^{\alpha_i}\|_t \\ \nonumber
    & \le
    C^D_{n,k}t^{-k}s^{-1}(1+A^{-1})\|c_{\alpha,j(\alpha)}
    x^{\alpha_{j(\alpha)}}\|_t.
  \end{align}
  Dividing by $c_{\alpha, j(\alpha)}$ and multiplying by
  $x^{\alpha-\alpha_{j(\alpha)}}$ we get the required decomposition.
\end{proof}

\begin{Lem}
  For any monomial $x^\alpha\in\frak{m}^{k+1}$ there exists a
  representation
  \begin{equation}\label{eq:decomposition to B}
    x^\alpha=\sum_{\bfi\in B} c_{\bfi,\alpha}x^\bfi+\sum 
    \tilde{u}_{i,\alpha}f_i+\tilde{E}_\alpha, \quad E_\alpha\in{\frak m}^{k+1}
  \end{equation}
  with
  \begin{align}
    \|\tilde{E}_\alpha\|_t&\le 
    2A^{-1}\|x^\alpha\|_t,\\
    \|\sum_{\bfi\in B} c_{\bfi,\alpha}x^\bfi\|_t&\le 
    (k+1)C^D_{n,k}C_{n,k}(A)^{-k}s^{-k-1}A^{-1} \|x^\alpha\|_t,\\
    \|\tilde{u}_{i,\alpha} 
    \|_t&\le 2^{n+1}C^D_{n,k}C_{n,k}(A)^{-k}s^{-k-1}\|x^\alpha\|_t,
  \end{align}
  where $A>2$ and $t$ is as in Lemma~\ref{lem:lem1ofDecomp}.
\end{Lem}

\begin{proof}
  Multiplying \eqref{eq:Monomial representation} by a monomial
  $x^\beta$, we get a similar representation for $x^{\alpha+\beta}$.
  If $|\beta|\ge k+1$ then the product
  $x^\beta\left(P_{\alpha}(x)+E'_\alpha\right)$ is already in
  $\frak{m}^{k+1}$ and the bounds follow trivially as soon as $A>2$.

  If, however, $|\beta|\le k$, then some monomials from the product
  could have degree smaller than $k+1$. Denote the sum of these
  monomials by $R$. Evidently, $\deg R\le k$ and
  $\|R\|_t<A^{-1}\|x^\alpha\|_t$. Using
  Proposition~\ref{prop:decomposition}, one can write
  \begin{equation}
    R=\sum_{\bfi\in B}
    c_\bfi x^\bfi + \sum U_i f_i + E, \quad E\in {\frak m}^{k+1}.
  \end{equation}
  We now express the bounds of Proposition~\ref{prop:decomposition} in
  terms of $\|\cdot \|_t$ norms. We have
  \begin{gather}
    \|\sum_{\bfi\in B}
    c_\bfi x^\bfi\|_t\le k\|\sum_{\bfi\in B}
    c_\bfi x^\bfi\|\le kC^D_{n,k}s^{-1} \|R\|\le 
    kC^D_{n,k}t^{-k}s^{-1}\|R\|_t,
    \\
    \|U_i\|_t\le 2^n\|U_i\|\le 2^nC^D_{n,k}t^{-k}s^{-1}\|R\|_t,
  \end{gather}
  and
  \begin{equation}\begin{aligned}
    \|E\|_t&\le t^{k+1}2^{k+n+1}\|E\|\le C^{D}_{n,k}2^{k+n+1}t^{k+1}s^{-1}\|R\| \\
    &\le  C^{D}_{n,k}2^{k+n+1}ts^{-1}\|R\|_t.
  \end{aligned}\end{equation}
  Together with bounds \eqref{eq: t bounds} and
  \eqref{eq:decomposition bounds} of Lemma~\ref{lem:lem1ofDecomp} this
  ends the proof.
\end{proof}
 
Extending \eqref{eq:decomposition to B} by linearity to $\frak{m}^k$,
we get the operator identity \eqref{eq:operator decomposition} with
operators $\pi_B, U$ bounded. Choosing $A=3$ and
$\epsilon'_{n,k}=C_{n,k}(3)\epsilon_{n,k}$, we get $\|E\|\le 2/3$ and
$\|(1-E)^{-1}\|\le 3$, which implies \eqref{eq:operator
  decomposition2} for $f\in{\frak m}^{k+1}$.

For arbitrary $f\in R_{n,t}$, we first decompose its $k$-th jet $j^k(f)$
as in Proposition~\ref{prop:decomposition}, and apply
\eqref{eq:operator decomposition2} to the remainder $E$. One can see
as above that $\|E\|_t\le C'_{n,k}\|f\|_t$ for some absolute constant
$C'_{n,k}$, and the result follows.
\end{proof}

\subsection{Lemma \`a la Valiron}

\begin{Lem}\label{lem:Valiron}
Let $\vec{a^j}=(a_{j,0},...,a_{j,k+1})$ be $N$ sequences of non-negative 
numbers satisfying
\begin{equation}\label{eq:Valiron condition}
a_{j,0}+...+a_{j,k}=1, \quad a_{j,k+1}\le M.
\end{equation}
For any $A>1, t_0>0$ there exist
$t\in(0,t_0]$, and 
$i(j), 0\le i(j)\le k$ for $j=1,...,N$ such that 
\begin{equation}\label{eq:choosing norm Lem}
t^{i(j)}a_{i(j)}\ge A\sum_{i\not= i(j)} t^ia_{j,i}\end{equation}
and
\begin{equation}\label{eq:Valiron bound on t}
t\ge \left(2A+1\right)^{-2N(k+1)}\min\left[t_0,\left(M(k+1)\right)^{-1}\right].
\end{equation}

\end{Lem}
\begin{proof} Let $B=2A+1$. Consider the functions
$\tilde{\phi_j}(i)=\log_Ba_{j,i}$ and let $\phi_j(i)$ be their
smallest  concave majorates, $j=1,...,N$.
From the bounds on $a_{j,i}$ we conclude that 
\begin{equation*}
\min_{i}\frac{{\phi_j}(k+1)-\phi_j(i)}{k-i}\le \log_B\left(M(k+1)\right).
\end{equation*}

Note that replacing $a_i$ by $t^ia_i$ adds the linear function $i\log_Bt$ to 
$\phi_j$. It is enough to find $t$  
such that the graphs of the piece-wise linear concave functions
$\phi_j(i)+i\log_Bt$ have maximum point $i(j)$ not at $k+1$ and
have edges  with absolute values of edge slopes at 
least $1$. Indeed, this slope  
condition means that  $t^ia_i\le B^{-|i-i(j)|}t^{i(j)}a_{i(j)}$, where $j$ is 
the maximum 
point of $\phi+k\log_Bt$, and \eqref{eq:choosing norm Lem} follows by the sum 
of geometric progression formula.

In other words, we have to find some  $t$ such that, first, $\log_Bt$ lies 
outside of 
$1$-neighborhood of  $-\Delta_{ji}$, where $\Delta_{ji}=\phi_j(i+1)-\phi_j(i)$ 
are slopes of the 
edges of all graphs of $\phi_j(i)$, and, second, $\log_B t$  is smaller than 
both $\log_B t_0$ and $-\Delta_{j,k}$ for all 
$j=1,...,N$. But this $1$-neighborhood has length at most $2N(k+1)$, and one 
can see from \eqref{eq:Valiron condition} that 
\begin{equation}
-\Delta_{j,k}\ge -\log_B\left(M(k+1)\right).
\end{equation}
so one can find the required $t$ with 
$$
\log_B t\ge \min\left(\log_B t_0, -\log_B\left(M(k+1)\right)\right)-2N(k+1)
$$
and  \eqref{eq:Valiron bound on t} follows.
\end{proof}

\section{Multiplicity operators, growth and zeros}
\label{sec:MOandZeros}

In this section we collect several theorems showing how the multiplicity
operators can be used to control the number of zeros and the asymptotic
growth of a map $F$ in small polydiscs or balls.

\subsection{Number of zeros in a small polydisc}

The following is a simple consequence of Lemma~\ref{lem:weierstrass}.

\begin{Thm} \label{thm:polydisc-zeros} Let $F\in\cO^n(D^n)$ with
  $\norm{F}\le1$. Assume $F$ has $k+1$ zeros (counted with
  multiplicities) in the polydisc $D_r^n$. Then for every $k$-th
  multiplicity operator $\mo$,
  \begin{equation}
    \CZ_{n,k} \abs{\mo_0(F)} \le r.
  \end{equation}
  where $\CZ_{n,k}>0$ is a universal constant.
\end{Thm}
\begin{proof}
  It is clearly sufficient to prove the statement for every basic
  multiplicity operator $\bmo_B^\alpha$. Let 
  $\CZ_{n,k}=\epsilon'_{n,k}$.

  Suppose that $\CZ_{n,k}s>r$. Making an
  arbitrarily small perturbation, we may also assume that $F$ admits
  $k+1$ \emph{distinct} zeros in $D^n_r$.

  Applying~Lemma~\ref{lem:weierstrass} we obtain $t>r$ such that
  the image of the restriction map $\cO(D^n)\to \cO(D^n_r)/\<F\>$ is at most
  $k$-dimensional (and generated by $x^B$). However, since $F$ admits
  $k+1$ distinct zeros in $D^n_r$, the restrictions
  of the polynomial functions already span a $k+1$ dimensional
  subspace of $\cO(D^n_r)/\<F\>$, leading to a contradiction.
\end{proof}

\subsection{Growth in small balls}

We begin this section with a simple result giving a lower bound for
univariate polynomials of unit norm. The following simple
lemma appears as \cite[Lemma 7]{polyfuchs}.

\begin{Lem} \label{lem:poly-discs-bound} Let $P=\sum_{i=0}^k c_iz^i$
  be a polynomial of degree $d$ and $\|P\|_1=\sum|c_i|=1$.  Let $Z=\{P=0\}$
  denote the zero-locus of $P$. Then for any $z\in D$,
  \begin{equation}
    \abs{P(z)} \ge C_{d} \dist(z,Z)^d
  \end{equation}
  where $C_d$ is some universal constant.
\end{Lem}

The following result is a local multidimensional analog of this lemma,
stated in terms of multiplicity operators.

\begin{Thm} \label{thm:sphere-growth} Let $F\in\cO^n(D^n)$ with
  $\norm{F}\le1$. Let $\mo$ be a multiplicity operator of order $k$.
  Assume $s=\abs{\mo_0(F)}\neq0$. There exist positive universal
  constants $A_{n,k},B_{n,k}$ with the following property:

  For every $r<s$ there exists $A_{n,k} r < \tilde r < r$ such that
  \begin{equation}
    \norm{F(z)} \ge B_{n,k} s \tilde r^k \mbox{ for every } \norm{z}=\tilde r
  \end{equation}
\end{Thm}
\begin{proof}
  Let $\ell$ denote one of the standard coordinates on $\C^n$.
  According to Lemma~\ref{lem:local-resultant} there exists a
  polynomial $P$, with  $\|P\|_1=1$, such that
  \begin{equation}
    P(\ell) = \sum_i U_i f_i + E
  \end{equation}
  where
  \begin{equation}
    \max(\norm{U_i},\norm{E})<\CD_{n,k} s^{-1} \mbox{ and }
    E=O(\abs{z}^{k+1}).
  \end{equation}
  Let $\rho = C_{n,k}^1r$ and $\tilde\rho=C_{n,k}^2r$ where the
  universal constants $C_{n,k}^2<C_{n,k}^1<1$ will be chosen later.
  
  From the Schwartz lemma it follows that for $\norm{z}<\rho$,
  \begin{equation}
    \begin{aligned}
      \abs{E(z)} &\le C^D_{n,k} s^{-1}
      \abs{z}^{k+1} \le \CD_{n,k}
      s^{-1}(C_{n,k}^1)^{k+1}r^{k+1}  \\
      &\le \CD_{n,k} (C_{n,k}^1)^{k+1}r^k
    \end{aligned}
  \end{equation}
  Let $U_{\tilde\rho}\subset\C$ denote the union of discs of radius
  $\tilde\rho$ around the zeros of $P$. By
  Lemma~\ref{lem:poly-discs-bound} we have
  \begin{equation}
    \abs{P(\ell(z))} \ge C_k \tilde\rho^k = C_k (C_{n,k}^2)^k r^k
    \mbox{ for } \ell(z)\not\in U_{\tilde\rho}.
  \end{equation}

  Now choose constants $C_{n,k}^2<C_{n,k}^1<1$ such that
  \begin{equation} \label{eq:rhotilde-choice}
    \begin{aligned}
      &\CD_{n,k} (C_{n,k}^1)^{k+1} < (1/2) C_k (C_{n,k}^2)^k \\
      &C_{n,k}^2 < (1/2) (4n)^{-1/2} k^{-n} C_{n,k}^1.
    \end{aligned}\end{equation}
  Then for any $z\in D^n_\rho$, if $\ell(z)\not\in U_{\tilde\rho}$ we
  have
  \begin{equation}
    \abs{\sum U_i(z) f_i(z)} = \abs{P(\ell(z))-E(z)} \ge (1/2) C_k \tilde\rho^k
  \end{equation}
  and since $\norm{U_i}\le\CD_{n,k}s^{-1}$ it follows that
  \begin{equation}
    \norm{F(z)} \ge (1/2) C_k \CD_{n,k} s \tilde\rho^k.
  \end{equation}

  To summarize, in $B^n_\rho$ we have
  $\norm{F(z)}\ge C_{n,k}^3 s \rho^k$ unless
  $\ell(z)\in U_{\tilde\rho}$ for some universal constant $C_{n,k}^3$.
  Since this is true for each coordinate $\ell$, we have in total a
  union $U$ of $k^n$ polydiscs of radius $\tilde\rho$ such that in
  $B^N_\rho\setminus U$ the inequality
  $\norm{F(z)}\ge C_{n,k}^3 s \rho^k$ holds.
  By~\eqref{eq:rhotilde-choice} $U$ has total diameter smaller than
  $\rho/2$, and hence there exists a sphere of radius
  $\rho/2 < \tilde r < \rho$ which is disjoint from $U$. This
  concludes the proof.
\end{proof}

In the proof of Lemma~\ref{lem:jet-mult} we used the fact that if
$\mult F \le k$ then ${\frak m}^{k+1}\subset I(F)$. In particular, this implies that if $G$
is a map whose components lie in ${\frak m}^{k+1}$ then $\mult F = \mult (F+G)$.
We can now give a quantitative version of this fact.

\begin{Cor} \label{cor:perts}
  Let $F,G\in\cO^n(D^n)$ with $\norm{F},\norm{G}\le1$.
  Let $\mo$ be a multiplicity operator of order $k$. Denote
  $s:=\abs{\mo_0(F)}\neq0$. Let $0<\e<1$. There exist positive
  universal constants $A'_{n,k},B'_{n,k}$ with the following property:

  If
  \begin{equation} \label{eq:k-jet-cond}
    \abs{\frac{1}{\alpha!} \pd{^\alpha G_i}{z^\alpha}(0)} \le \e^{k+1-\abs{\alpha}} \qquad \text{for all } i=1,\ldots,n,\ \abs{\alpha}\le k
  \end{equation}
  then for every $\e<r<B'_{n,k}s$ there exists $A'_{n,k} r < \tilde r < r$ such that
  \begin{equation}
    \norm{F(z)} > \norm{G(z)} \text{ for every } \norm{z}=\tilde r,
  \end{equation}
  and in particular
  \begin{equation}
    \#\{z:F(z)=0, \norm{z}<\tilde r\} = \#\{z:F(z)+G(z)=0, \norm{z}<\tilde r\}.
  \end{equation}
\end{Cor}
\begin{proof}
  Indeed, by Theorem~\ref{thm:sphere-growth} there exists $A_{n,k} r < \tilde r < r$
  such that
  \begin{equation}
    \norm{F(z)} \ge B_{n,k} s \tilde r^k \text{ for every } \norm{z}=\tilde r.
  \end{equation}
  On the other hand we have
  \begin{equation}
    G(z) = \sum_{\abs{\alpha}\le k} \frac{1}{\alpha!} \pd{^\alpha G_i}{z^\alpha}(0) z^\alpha + G_{k+1}(z)
  \end{equation}
  where $\norm{G_{k+1}}<C^1_{n,k}$. From~\eqref{eq:k-jet-cond} and the Schwartz lemma applied
  to $G_{k+1}$ we conclude that
  \begin{equation}
    \norm{G(z)} \le C^2_{n,k} \tilde r^{k+1} \text{ for every } \norm{z}=\tilde r.
  \end{equation}
  Choosing an appropriate $B'_{n,k}$ we can guarantee that
  $\norm{F(z)}>\norm{G(z)}$ for every $\norm{z}=\tilde r$, and the
  conclusion follows by an argument of Rouch\'e type (using
  topological degree theory).
\end{proof}

We also record a convenient special case of Corollary~\ref{cor:perts}.

\begin{Cor} \label{cor:power-pert} Let $F,G\in\cO^n(D^n)$ with
  $\norm{F},\norm{G}\le1$. Let $\mo$ be a multiplicity operator of
  order $k$. Denote $s:=\abs{\mo_0(F)}\neq0$. Let $0<\e<1$. There
  exist universal constants $A''_{n,k},B''_{n,k}$ with the following
  property:

  If $\norm{G(0)}<\e$ then for every $\e<r<B''_{n,k}s$ there exists
  $A''_{n,k} r < \tilde r < r$ such that
  \begin{equation}
    \norm{F(z)} > \norm{G^{k+1}(z)} \text{ for every } \norm{z}=\tilde r,
  \end{equation}
  and in particular
  \begin{equation}
    \#\{z:F(z)=0, \norm{z}<\tilde r\} = \#\{z:F(z)+G^{k+1}(z)=0, \norm{z}<\tilde r\}.
  \end{equation}
\end{Cor}
The proof is analogous to the proof of Corollary~\ref{cor:perts}.

\section{Multiplicity operators and ideals}
\subsection{Growth along analytic curves}

Let $\gamma\subset(\C^n,0)$ be a germ of a real analytic curve.  Let
$G(t):(\R_{\ge0},0)\to\C^n$ be a length parametrization for $\gamma$,
\begin{equation}
  \gamma = \{G(t):t\in(\R,0)\} \qquad \norm{G(t)} = t.
\end{equation}
The components of $G(t)$ are convergent Puiseux series in
$t$.

We define the \emph{order} of an analytic function along $\gamma$ as
\begin{equation}
  \ord_\gamma : \cO(\C^n,0) \to \Q_{\ge0}
  \qquad \ord_\gamma(f)=\lim_{t\to0^+} \frac{\log \abs{f(G(t))}}{\log t}.
\end{equation}
In terms of Puiseux series, the order of $f$ along $\gamma$ is equal
to the leading exponent in the Puiseux expansion of the function
$f(G(t))$.

We will require the following standard lemma.
\begin{Lem}\label{lem:curve-mult}
  Let $k\in\N$. Suppose that $\ord_\gamma f_i \ge k$. Then $\mult F\ge
  k$.
\end{Lem}
\begin{proof}
  Let $\ell$ be a functional on $\C^n$ which is transversal to the
  tangent vector to $\gamma$ at $t=0$. Then one can easily check that
  the functions $1,\ell,\ldots,\ell^{k-1}$ are linearly independent
  (over $\C$) in the local ring of $F$.  Hence the dimension of the
  local ring is at least $k$.
\end{proof}

We now examine the order of a multiplicity operator along an analytic
curve.
\begin{Thm} \label{thm:curve-growth} Let $\mo$ be a
  multiplicity operator of order $k$. Then
  \begin{equation}
    \ord_\gamma \mo(F) \ge \min\{\ord_\gamma f_i: i=1,\ldots,n\} - k
  \end{equation}
\end{Thm}
\begin{proof}
  In this proof it is more convenient to let $t$ denote a Puiseux
  coordinate (i.e. a coordinate transversal to $\gamma$ at the origin)
  rather than the arc-length parameter. This does not change the
  notion of the order along $\gamma$.
  
  Let $\tau\in(\R_{\ge0},0)$ and consider the functions
  \begin{equation}
    \begin{aligned}
      f_1^\tau &= f_1 - E_1^\tau \qquad E_1^\tau =
      f_1(G(\tau))+\ldots+\frac1{k!}\frac\d{\d t^k}f_1(G(t))\vert_{t=\tau} 
      (t-\tau)^k \\
      &\vdots \\
      f_n^\tau &= f_n - E_n^\tau \qquad E_n^\tau =
      f_n(G(\tau))+\ldots+\frac1{k!}\frac\d{\d
        t^k}f_n(G(t))\vert_{t=\tau} (t-\tau)^k
    \end{aligned}
  \end{equation}
  If we let $\gamma_\tau$ denote the germ of $\gamma$ through the
  point $G(\tau)$, then it is clear by construction that
  $\ord_{\gamma_\tau} f_i^\tau> k$ for $i=1,\ldots,n$. Therefore by
  Lemma~\ref{lem:curve-mult} we have $\mult F^\tau\ge k$, where 
  $F^\tau=(f_1^\tau,...,f_n^\tau)$, and hence
  $\mo(F^\tau)\rest{G(\tau)}=0$.

  Since $\mo$ depends polynomially on the $k$-jet of its
  parameter, its derivative is uniformly bounded by a constant $C$
  around the $k$-jet $j(F)\in \left(J_{n,k}(G(\tau))\right)^n$ of $F$ at 
  $G(\tau)$.  Thus for every sufficiently small $\tau$,
  \begin{equation}
    \begin{aligned}
      \abs{\mo(F)\vert_{G(\tau)}} &= \abs{\mo(F^\tau+E^\tau)\vert_{G(\tau)}} \\
      &\le \abs{\mo(F^\tau)\vert_{G(\tau)}} + C \norm{E^\tau} =
      C\norm{E^\tau}
    \end{aligned}
  \end{equation}
  Finally, since $E^\tau$ is defined in terms of derivatives up to
  order $k$ of $f_i(G(t)),i=1,\ldots,n$ and since each derivative can
  decrease the order by at most $1$, we conclude that
  \begin{equation}
    \ord_\gamma \norm{E^t} \ge \min\{\ord_\gamma f_i: i=1,\ldots,n\} - k
  \end{equation}
  as claimed.
\end{proof}

\subsection{Integral closure and Multiplicity}
\label{sec:iclos-mult}

We recall the notions of \emph{multiplicity} and \emph{integral
  closure} of an ideal and some known results concerning them.

Let $p$ be a point in $\C^n$ and denote by $\cO_p$ the local ring of
holomorphic functions in a neighborhood of $p$. We denote by $\frak m$
the maximal ideal of $\cO_p$. Let $I\subset \cO_p$ be an $\frak
m$-primary ideal. The \emph{Hilbert-Samuel multiplicity} of $I$ on
$\cO_p(\C^n,0)$ is defined to be
\begin{equation}
  \mult I = d! \lim_{n\to\infty} \frac{\lambda(\cO_p/I^n)}{n^d}
\end{equation}
where $\lambda$ denotes the length as an $\cO_p$-module. When $I$ is
generated by a regular sequence, $I=\<f_1,\ldots,f_n\>$ this notion
agrees with the geometric notion of $\mult(f_1,\ldots,f_n)$
\cite[Proposition 11.1.10]{hune:ic}.

The Hilbert-Samuel multiplicity satisfies a Brunn-Minkowski type
inequality (\cite[Appendix by Tessier]{el:multiplicities}, see also
\cite{kk:local-mult}). Namely, for any two $\frak m$-primary ideals
$I,J\subset \cO_p$ we have
\begin{equation} \label{eq:brunn-minkowski} \mult^{1/n}(IJ) \le
  \mult^{1/n}I + \mult^{1/n}J
\end{equation}

Recall that the integral closure $\iclo I$ of an ideal $I\subset\cO_p$
is the ideal generated by all elements $x\in\cO_p$ satisfying an
equation $x^n+a_1x^{n-1}\ldots+a_n=0$ with $a_k\in I^k$. By a theorem
of Rees, we have $\mult I = \mult\iclo I$ \cite[Theorem 11.3.1]{hune:ic}.

Finally recall the following characterization of the integral closure
(see \cite[Theorem 2.1]{Teis:ic} or \cite[Theorem 5.A.1]{milman:poincare-metric}).
\begin{Thm} \label{thm:iclo-criterion} Let $f_1,\ldots,f_q$ be
  generators of $I$ and $f\in\cO_p$.  Then $f\in\iclo I$ if and only
  if for any germ of an analytic curve $\gamma:(\C,0)\to(\C^n,p)$ the
  ratio
  \begin{equation}
    \frac{f(\gamma(t))}{\max\{\abs{f_1(\gamma(t))},\ldots,\abs{f_q(\gamma(t))}\}}
  \end{equation}
  remains bounded as $t\to0$.
\end{Thm}

To begin our study of the relation between the multiplicity operators
and the notions of multiplicity and integral closure we make the
following definition.
\begin{Def}
  The \emph{$k$-th multiplicity operator ideal} of $I$ is defined to
  be
  \begin{equation}
    M_{n,k}(I) = I+\<\mo(f_1,\ldots,f_n)\>
  \end{equation}
  where $\mo$ varies over the multiplicity operators of order
  $k$, and $(f_1,\ldots,f_n)$ varies over all $n$-tuples of elements
  in $I$.
\end{Def}

We now present a result relating the multiplicity operator ideal to
the integral closure and multiplicity of an ideal.

\begin{Thm} \label{thm:mult-op-iclo} Let $I\subset\cO_p$ be an
  ideal. Then
  \begin{equation} \label{eq:mult-op-iclo} {\frak m}^k M_{n,k}(I)
    \subset \iclo I.
  \end{equation}
  In particular, if $I$ is $\frak m$-primary then
  \begin{equation}
    \mult (I+{\frak m}^k M_{n,k}(I)) = \mult I.
  \end{equation}
\end{Thm}
\begin{proof}
  This follows immediately from Theorems~\ref{thm:curve-growth}
  and~\ref{thm:iclo-criterion}.
\end{proof}

And the following corollary.

\begin{Cor} \label{cor:mult-up-ideal-mult} Let $I\subset\cO_p$ be an
  $\frak m$-primary ideal. Then
  \begin{equation}
    \mult^{1/n} M_{n,k}(I) \ge \mult^{1/n} I - k
  \end{equation}
\end{Cor}
\begin{proof}
  By~\eqref{eq:mult-op-iclo} we have
  \begin{equation}
    \mult ({\frak m}^k M_{n,k}(I)) \ge \mult\iclo I = \mult I.
  \end{equation}
  On the other hand, by~\eqref{eq:brunn-minkowski} we have
  \begin{equation}
    \begin{aligned}
      \mult^{1/n} ({\frak m}^k M_{n,k}(I)) &\le \mult^{1/n}{\frak m}^k + \mult^{1/n} M_{n,k}(I) \\
      &= k + \mult^{1/n} M_{n,k}(I),
    \end{aligned}
  \end{equation}
  and the conclusion easily follows.
\end{proof}

\section{Applications for Noetherian Functions}

A subring $K\subset\cO(U)$ for a domain $U\subset\C^n$ is called a
\emph{ring of Noetherian functions} if
\begin{enumerate}
\item It contains the polynomial ring $\C[x_1,\ldots,x_n]$ and is
  finitely generated as a ring over it.
\item It is closed under the partial derivatives
  $\pd{}{x_i},i=1,\ldots,n$.
\end{enumerate}
This notion is due to Khovanskii. A related notion where the
assumption of finite generation is replaced by a topological noetherianity
condition was introduced by Tougeron in \cite{tougeron:noetherian}.

More explicitly, we consider the polynomial ring
$R=\C[x_1,\ldots,x_n,f_1,\ldots,f_m]$ along with the system of
differential equations
\begin{equation} \label{eq:noetherian-system} \pd{f_i}{x_j} = P_{ij},
  \qquad P_{ij}\in R \quad \mbox{for}\quad i=1,\ldots,n,\quad
  j=1,\ldots,m.
\end{equation}
For simplicity we will assume that this system is integrable
everywhere, and therefore $\spec R$ is foliated by graphs of solutions
of satisfying~\eqref{eq:noetherian-system}.
\begin{Rem}
  In general it may happen that the locus of integrability of the
  system forms only an algebraic subvariety. This extra generality is
  studied in \cite{GabKho} and our arguments can be extended using
  their approach. We restrict attention to the completely integrable
  case to simplify the presentation.
\end{Rem}

For every $p\in\spec R$ we denote by $\cL_p$ the graph of the solution
of the system~\eqref{eq:noetherian-system} around the point $p$. On
$\cL_p$ we may think of $f_1,\ldots,f_m$ as functions of the variables
$x_1,\ldots,x_n$.

Let $p\in\spec R$ and let $P_1,\ldots,P_n\in R$ be an $n$-tuple of
polynomials. Suppose that $\deg P_{ij}\le\delta$ and $\deg P_i\le d$.
If
\begin{equation}
  \mult_p (P_1\rest{\cL_p},\ldots,P_n\rest{\cL_p}) < \infty,
\end{equation}
then it is natural to ask for an explicit upper bound for this
multiplicity in terms of the parameters $n,m,d,\delta$. This problem
was studied in \cite{GabKho} using topological methods related to the
study of Milnor fibers, leading to the following estimate

\begin{Thm} \label{thm:gk-mult}
  The multiplicity $\mult_p
  (P_1\rest{\cL_p},\ldots,P_n\rest{\cL_p})$, assuming it is finite,
  does not exceed the maximum of the following two numbers:
  \begin{eqnarray*}
    &\frac{1}{2}Q\((m+1)(\delta-1)[2\delta(n+m+2)-2m-2]^{2m+2}
      +2\delta(n+2)-2\)^{2(m+n)} \\
    &\frac{1}{2}Q\(2(Q+n)^n(d+Q(\delta-1))\)^{2(m+n)},
    \quad \text{where} \quad Q=en\(\frac{e(n+m)}{\sqrt{n}}\)^{\ln n+1}
      \(\frac{n}{e^2}\)^n
  \end{eqnarray*}
\end{Thm}

In subsection~\ref{sec:noetherian-mult} we present a simple approach
to this problem using multiplicity operators. While the estimates
obtained by this method are not as sharp as those of \cite{GabKho}, the
proof is simple and illustrates the manner in which multiplicity
operators can be used in the study of multiplicities.

Having obtained a bound on the multiplicity of a Noetherian
intersection, it is natural to move beyond the local context, and ask
whether one can estimates the number of solutions of a system of
equations $P_1\rest\cL=\ldots=P_n\rest\cL=0$ on a compact piece of the
graph $\cL$.  In subsection~\ref{sec:noetherian-semilocal} we use the
multiplicity operators to prove a result in this direction.

\subsection{An estimate for the intersection multiplicity of
  Noetherian functions}
\label{sec:noetherian-mult}

For any point $p\in\spec R$, one can consider the polynomials
$P_1,\ldots,P_n$ as functions of $x_1,\ldots,x_n$ by restricting them
to the graph $\cL_p$. Then their derivatives with respect to
$x_1,\ldots,x_n$ are again given by polynomial functions according
to~\eqref{eq:noetherian-system}.

Let $\mo$ be a $k$-th multiplicity operator in the variables
$x_1,\ldots,x_n$.  By the above, it follows that
$\mo(P_1,\ldots,P_n)$ (viewed, at any point $p\in\spec R$, as
the multiplicity operator of $P_1\rest{\cL_p},\ldots,P_n\rest{\cL_p}$)
is a polynomial, and moreover
\begin{equation} \label{eq:nf-deg-bound} \deg
  \mo(P_1,\ldots,P_n) \le {n+k \choose k} (d+k \delta).
\end{equation}

The following lemma shows that if an algebraic variety meets a graph
$\cL_p$ in an isolated point, then generic nearby intersections are
transversal.

\begin{Lem} \label{lem:transversality} Let $V\subset\spec R$ be an
  algebraic variety, and $p\in V$ a point in $V$. Suppose that
  $V\cap\cL_p=\{p\}$ in a small neighborhood of $p$. Then at a generic
  point $q\in V$ near $p$, the graph $\cL_q$ is transversal to $V$.
\end{Lem}
\begin{proof}
  Let $D$ be a small ball around $p$, and consider the projection
  $\pi:(D\cap V)\to(\C^m,0)$ from $D\cap V$ to the $m$-dimensional
  space parameterizing the leafs around $\cL_p$. By assumption $\pi$ is
  a proper finite-to-one map, and hence $\dim D\cap V=\dim\Im\pi$.
  Then by Sard's theorem, a generic point $q\in D\cap V$ is a
  non-critical point of $\pi$, which means that $V$ is transversal to
  $\cL_q$ at $q$ as stated.
\end{proof}

We now show that when the variety of an ideal $I$ meets a graph
$\cL_p$ transversally, the multiplicity of $I\rest{\cL_p}$ at $p$ can
be bounded in algebraic terms.

\begin{Lem} \label{lem:transverse-mult-bound} Let
  $I=\<F_1,\ldots,F_k\>\subset R$ and $p\in V(I)$ a point where $V(I)$
  is smooth of codimension $k$ and transversal to $\cL_p$. Suppose
  $\deg F_i\le D$ for $i=1,\ldots,k$. Then
  \begin{equation}
    \mult_p I\rest{\cL_p} \le D^{kn}.
  \end{equation}
\end{Lem}
\begin{proof}
  Let $J=\sqrt{I}$. Since $V(I)$ is smooth at $p$, $J$ contains
  functions whose differentials define the tangent space to $V(I)$ at
  $p$. Since this tangent space is transversal to $\cL_p$, one can
  find $n$ functions in $J$ whose differentials are linearly
  independent on $\cL_p$. Therefore $\mult J\rest{\cL_p}=1$.

  It remains to note that by the effective Nullstellensatz of
  \cite{kollar:sen}, $J^{D^k}\subset I$, and therefore
  \begin{equation}
    \mult I\rest{\cL_p} \le \mult {J^{D^k}} = D^{kn} \mult J = D^{kn} .
  \end{equation}
\end{proof}

We are now ready to state and prove an upper bound for intersection
multiplicities of Noetherian functions.

\begin{Thm} \label{thm:noetherian-mult-bound} Let $p\in\spec R$ and
  let $P_1,\ldots,P_n\in R$ be an $n$-tuple of polynomials with $\deg
  P_i\le d$.  Assume
  \begin{equation}
    \mult_p (P_1\rest{\cL_p},\ldots,P_n\rest{\cL_p}) < \infty.
  \end{equation}
  Then
  \begin{equation}
    \mult_p (P_1\rest{\cL_p},\ldots,P_n\rest{\cL_p}) < (2d\delta)^{n (n+1)^{2m}(m+n)^m}.
  \end{equation}
\end{Thm}
\begin{proof}
  Let $I_0=\<P_1,\ldots,P_n\>$ be the ideal generated by
  $P_1,\ldots,P_n$ and $V_0$ the corresponding variety. By the
  assumption, $V_0$ has codimension $k=n$ around $p$. By
  Lemmas~\ref{lem:transversality} and~\ref{lem:transverse-mult-bound},
  at a generic point $q$ close to $p$ we have
  \begin{equation}
    \mult_q I_0\rest{\cL_q} \le d^{nk} < d^{n(m+n)}.
  \end{equation}
  Let $M_0$ denote a generic multiplicity operator of order
  $d^{n(m+n)}$ of $I_0$. Then $M_0$ does not vanish at generic points
  $q$ close to $p$. By~\eqref{eq:nf-deg-bound},
  \begin{equation}
    \deg M_0 \le {n+d^{n(m+n)} \choose d^{n(m+n)}}(d+d^{n(m+n)}\delta)
  \end{equation}
  Let $\ell\in R$ denote a generic affine functional vanishing on
  $p$. Denote
  \begin{equation}
    M_0' = \ell^{d^{n(m+n)}} M_0
  \end{equation}
  and $I_1=I_0+\<M_0'\>$. Then
  \begin{equation} \label{eq:rprime-deg-bound} \deg M_0' \le
    2d^{(n+1)^2(m+n)}\delta,
  \end{equation}
  and by Theorem~\ref{thm:mult-op-iclo} we have
  \begin{equation}
    \mult_p I_0\rest{\cL_p} = \mult_p I_1\rest{\cL_p}.
  \end{equation}

  Since $M_0'$ does not vanish at generic points of $V_0$ near $p$,
  $V(I_1)$ has codimension $k=n+1$ in a neighborhood of $p$. One can
  then repeat the argument above to obtain a sequence of ideals
  $I_0\subset\cdots\subset I_m$, where $V(I_k)$ has codimension $n+k$
  and
  \begin{equation}
    \mult_p I_0\rest{\cL_p} = \cdots = \mult_p I_m\rest{\cL_p}.
  \end{equation}
  By induction using~\eqref{eq:rprime-deg-bound}, all generators of
  $I_m$ have degrees bounded by
  \begin{equation}
    \deg M_m' \le (2d\delta)^{(n+1)^{2m}(m+n)^m}.
  \end{equation}
  Finally, $I_m$ is a zero-dimensional complete intersection near $p$,
  and from the Bezout theorem we deduce that
  \begin{equation}
    \mult_p I_0\rest{\cL_p} = \mult_p I_m\rest{\cL_p}\le \mult_p I_m \le (2d\delta)^{n (n+1)^{2m}(m+n)^m}
  \end{equation}
  as claimed.
\end{proof}

\subsection{An estimate for the number of solutions of a Noetherian
  system}
\label{sec:noetherian-semilocal}

Let $\cP$ denote an algebraic family of Noetherian equations in $n$
variables and $m$ functions in the origin, i.e. an algebraic variety
where each point consists of a tuple of coefficients $\{P_{ij}\}$ of a
Noetherian system~\eqref{eq:noetherian-system}, plus a tuple of
polynomials $\{P_i\}$. We think of each point in the variety as
representing the systems of equations $\{P_i=0\}$ on the solution
of~\eqref{eq:noetherian-system} through the origin, $\cL_0$.
\begin{Rem}
  The choice of the origin is a matter of technical convenience. One
  can of course add the coordinates of an arbitrary point to the
  description of $\cP$.
\end{Rem}

Let $\cP_\Sigma\subset\cP$ denote the subvariety of equations having a
non-isolated solution at the origin, i.e. such that
\begin{equation}
  \mult_0 (P_1\rest{\cL_p},\ldots,P_n\rest{\cL_p}) = \infty.
\end{equation}
Let $K$ denote an upper bound for the multiplicity of an isolated
Noetherian intersection in the class $\cP$. For instance, one could
use the estimate given by Theorem~\ref{thm:noetherian-mult-bound} or
the better estimates given in \cite{GabKho}. For specific classes $\cP$ one
may have better a-priori bounds $K$.

Let $I_K$ denote the ideal generated by all multiplicity operators of
order $K$ of the tuple $P_1,\ldots,P_n$. Then the common zeros of
$I_K$ correspond to Noetherian intersections of multiplicity greater
than $K$, and by the definition of $K$ this implies
$V(I_K)=\cP_\Sigma$.

\begin{Thm} \label{thm:noetherian-semilocal} Suppose that $\cP$ is
  defined as an affine subvariety of $\C^N$ by equations of degree
  bounded by $D$. Let $\norm{\cdot}$ denote the usual Euclidean norm
  for some choice of linear coordinates in $\C^N$ (the choice of
  different norms will only affect the constant $C$ below). Let $K$ denote
  an upper bound for the multiplicity of an isolated Noetherian
  intersection in $\cP$.

  Let $p\in\cP$, and suppose that the corresponding Noetherian
  functions $P_1,\ldots,P_n$ are defined in the unit ball and have
  norm bounded by $1$. Then the number of solutions of the Noetherian
  system $P_1,\ldots,P_n=0$ in a ball of radius $r$ around the origin
  (in the $x$ variables) does not exceed $K$, where
  \begin{equation}
    r = C \(\frac{\dist(p,\cP_\Sigma)}{1+\norm{p}^2}\)^e, \quad  e = \max\(D, 
    {n+K \choose k}(d+K\delta)\)^N
  \end{equation}
  and $C$ is a constant depending only on $\cP$.
\end{Thm}
\begin{proof}
  By the effective \Lojas. inequality given in \cite{JKS92}, there
  exists a constant $C'>0$ such that
  \begin{equation}
    \max_{\alpha,\abs{B}=K} \abs{\mo_p(P_1,\ldots,P_n)} \ge  C' 
    \(\frac{\dist(p,\cP_\Sigma)}{1+\norm{p}^2}\)^{e}.
  \end{equation}
  The result now follows by applying Theorem~\ref{thm:polydisc-zeros}.
\end{proof}

\begin{Rem}
  One can estimate the norms of $f_1,\ldots,f_m,$ and hence
  $P_1,\ldots,P_n,$ in terms of the coefficients of the Noetherian
  system (in some fixed ball) using Gronwall's inequality. In some
  cases one may have better a-priori estimates due to the structure of
  the particular Noetherian system.
  
  When the variety $\cP$ is defined by equations involving only
  rational coefficients, one can also obtain estimates for the
  constant $C$ in terms of the complexity of the description of the
  system by using the appropriate effective \Lojas. inequality
  results.

  We leave the details of these explicit computations for the reader.
\end{Rem}

\subsection{Concluding remarks}

The arguments of subsection~\ref{sec:noetherian-mult}, as well as the
finer arguments of \cite{GabKho}, establish upper bounds for multiplicities
of isolated solutions of Noetherian equations. They do not apply when
studying degenerate systems with non-isolated solution. This
limitation is reflected in Theorem~\ref{thm:noetherian-semilocal} as
we demonstrate below.

Recall that we estimate the number of solutions in a ball of a
specified radius, and this radius tends to zero as $p$ tends to
infinity or to $\cP_\Sigma$.  As $p$ tends to infinity, one can indeed
construct examples where the number of zeros in a ball of any fixed
radius also tends to infinity. For instance one can consider the
system
\begin{align}
  \pd f x &= p g \\
  \pd g x &= -p f
\end{align}
where $p\in\C$, with the solution $f=\sin(px),g=\cos(px)$, along with
the equation $f=0$.

On the other hand, we have no similar examples corresponding to the
case where $p$ tends to $\cP_\Sigma$. In fact, it is an old conjecture
of Khovanskii (which is given precisely in \cite{GabKho}) that this
scenario is impossible. However, all known results apply only to
multiplicities of isolated solution --- and this is reflected in our
estimates, which become weaker and weaker as we approach the locus of
non-isolated solutions. 

The study of deformations of systems of Noetherian equations with
non-isolated solutions has been one of our principal motivations in
developing the theory of multiplicity operators. Our intention in this
section has been to develop a simple representative application of
this theory in the context of Noetherian functions. A more systematic
application for the study of of non-isolated solutions will appear in
a separate paper.

\bibliographystyle{plain}
\bibliography{nrefs}

\end{document}